\documentclass[11pt]{article}
\usepackage{amssymb, amsmath, amsthm}
\usepackage[dvips]{graphicx}

\newtheorem{thm}{Theorem}[section]
\newtheorem{lem}[thm]{Lemma}
\newtheorem{cor}[thm]{Corollary}
\newtheorem{prop}[thm]{Proposition}

\newtheorem{rem}{Remark}[section]
\newtheorem{exa}{Example}[section]

\numberwithin{equation}{section}

\renewcommand{\a}{\alpha}
\renewcommand{\b}{\beta}
\newcommand{\e}{\varepsilon}

\newcommand{\la}{\lambda}

\newcommand{\Si}{\Sigma}

\newcommand{\De}{\Delta}

\def\R{{\mathbb{R}}}
\def\N{{\mathbb{N}}}
\def\Z{{\mathbb{Z}}}
\def\T{{\mathbb{T}}}
\def\C{{\chi}}

\allowdisplaybreaks

\title{Hydrodynamic limit for particle systems with 
degenerate rates without 
exclusive constraints}
\author{Makiko Sasada}
\date{}

\begin{document}
\maketitle

\begin{abstract}
\noindent
We consider the hydrodynamic behavior of some conservative particle systems with degenerate jump rates without exclusive constraints. More precisely, we study the particle systems without restrictions on the total number of particles per site with nearest neighbor exchange rates which vanish for certain configurations. Due to the degeneracy of the rates, there exists blocked configurations which do not evolve under the dynamics and all of the hyperplanes of configurations with a fixed number particles can be decomposed into different irreducible sets. We show that, for initial profiles smooth enough and bounded away from zero, the macroscopic density profile evolves under the diffusive time scaling according to a nonlinear diffusion equation (which we call the modified porous medium equation). The proof is based on the Relative Entropy method but it cannot be straightforwardly applied because of the degeneracy. 
\footnote{
Graduate School of Mathematical Sciences,
The University of Tokyo, Komaba, Tokyo 153-8914, Japan.}
\footnote{ 
$\quad$ e-mail: sasada@ms.u-tokyo.ac.jp, 
Fax: +81-3-5465-7011.}
\footnote{
\textit{The author is a JSPS Research Fellow and supported by the JSPS Grant 21-3656.}}
\end{abstract}
\section{Introduction}

Gon\c{c}alves, Landim and Toninelli established the hydrodynamic limit for some particle systems with degenerate rates under exclusive constraints in \cite{GLT}. They showed that the macroscopic density profile for their model evolves under the diffusive time scaling according to the porous medium equation.

In this paper, we consider some particle systems on the $d$-dimensional torus $\T^d_N$ with degenerate rates without restrictions on the total number of particles per site to obtain a microscopic derivation of the modified porous medium equation defined below.

The modified porous medium equation (MPME) is a partial differential equation of the form
\begin{equation} \label{eq:mpme}
\left\{
\begin{aligned}
\partial_t\rho(t,u) & = \De(\Phi(\rho(t,u))^m) \\
\rho(0,.) & = \rho_0(.)
\end{aligned}
\right.
\end{equation}
where $\De = \Si_{1 \le j \le d}\partial^2_{u_j}$, $m \in \N \setminus \{0,1\}$ and $\Phi(\rho)$ is a smooth strictly increasing function satisfying $\Phi(0)=0$ and $\lim_{\rho \to 0}\Phi^{\prime}(\rho) < \infty$. This can be rewritten in the divergence form as $\partial_t\rho(t,u)=\nabla(D(\rho(t,u))\nabla(\rho(t,u)))$ with diffusion coefficient $D(\rho(t,u))=m\Phi(\rho)^{m-1} \Phi^{\prime}(\rho)$. Note that $D(\rho)$ goes to zero as $\rho \to 0$, thus the equation looses its parabolic character.

To obtain a microscopic derivation of the MPME, we study stochastic particle systems on the $d$-dimensional discrete torus $\T^d_N$ without restrictions on the total number of particles per site. A configuration space of our microscopic dynamics is therefore given by $\N^{\T^d_N}$ with $\N=\{0,1,2,...\}$ and a configuration is defined by giving for each site $x \in\T^d_N$ the occupation variable, $\eta(x) \in \N$, which stands for the total number of particles at $x$. The process is defined through a function $g:\N \to \R_+$ vanishing at zero as follows. The evolution of our system is a continuous time Markov process in which each particle jumps from site $x$ to a nearest neighbor site $y$ at a rate $c(x,y,\eta)g(\eta(x))(\eta(x))^{-1}$. Namely, if there are $k$ particles at a site $x$, at rate $c(x,y,\eta)g(\eta(x))$ one of the particles at $x$ jumps to $y$. For each $m \in \N \setminus \{0,1\}$, we can provide a proper choice of $c(x,y,\eta)=c(y,x,\eta)$ to derive the MPME with the correspondent $m$ (see (\ref{eq:m2jump}) for $m=2$ and (\ref{eq:m3jump}) for $m=3$ in the next section). The function $\Phi(\rho)$ appearing in the hydrodynamic equation is given as an expectation value of $g$ with respect to an invariant measure $\nu_{\rho}$, which is defined in the next section and parameterized by the density of particles (see (\ref{eq:phidef}) in the next section). We remark that the choice $c(x,y,\eta)=1$ corresponds to the Zero Range process and, as is well known, leads to the nonlinear heat equation with $D(\rho)=\Phi^{\prime}(\rho)$ under diffusive re-scaling of time (see e.g. Section 5 and 6 in \cite{KL}). For a technical reason, in addition to an assumption usually assumed for $g$ to obtain the hydrodynamic behavior of the Zero Range process, we have to assume another condition for $g$ called $(G)$. As we note at Remark 2.4, the condition $(G)$ depends on $m$.

This paper is organized as follows: In Section 2 we introduce our model and state the main result. In Section 3, we give some examples for $g$ satisfying the desired condition. In Section 4, we give the proof of the main theorem via the Relative Entropy method. The proof of One block estimate and Proposition \ref{prop:convex} needed for the Relative Entropy method are postponed to Section 5 and Section 6, respectively. 
 
\section{Notation and Results}
\noindent
We consider the continuous time Markov process $\eta_t$ with state space $\chi^d_N=\N^{\T^d_N}$, where $\T^d_N=\{0,1,...,N-1\}^d$ is the discrete $d$-dimensional torus. Let $\eta$ denote a configuration in $\chi^d_N$, $x$ a site in $\T^d_N$ and $\eta(x)=k$ if there are $k$ particles at site $x$. The elementary moves which occur during evolution correspond to jump of particles among nearest neighbors, $x$ and $y$, occurring at a rate $c(x,y,\eta)$ times $g(\eta(x))$ where a function $g:\N \to \R_+$ satisfies that $g(k)=0$ if and only if $k=0$. Here, $c(x,y,\eta)=c(y,x,\eta)$ depends both on the couple $(x,y)$ and on the value of the configuration $\eta$ in a finite neighborhood of $x$ and $y$. On the other hand, $g(\eta(x))$ depends only on the value of the configuration $\eta$ at site $x$. Precisely, the dynamics is defined by means of an infinitesimal generator acting on cylinder functions $f:\chi^d_N \to \R$ as
\begin{displaymath}
(L_Nf)(\eta) = \sum_{x,y \in \T^d_N,|x-y|=1 }c(x,y,\eta)g(\eta(x))(f(\eta^{x,y})-f(\eta)),
\end{displaymath}
where $|x-y| = \sum_{1\le i \le d}|x_i-y_i|$ is the sum norm in $\R^d$ and
 \begin{equation}
 \eta^{x,y}(z) = \begin{cases}
	\eta(z)  & \text{if $z \neq x,y$} \\
	\eta(x)-1 & \text{if $z = x$} \\
	\eta(y)+1  & \text{if $z = y$}.
\end{cases}
\end{equation}

In the sequel we consider the rates 
\begin{equation}\label{eq:m2jump}
c(x,x+e_j,\eta)=g(\eta(x-e_j))+g(\eta(x+2e_j))
\end{equation}
where $\{e_j,j=1,...,d\}$ denotes the canonical basis of $\R^d$ and we will prove all the theorems for this choice. This, as we will prove, leads in the hydrodynamic limit to the modified porous medium equation (\ref{eq:mpme}) for $m=2$. Since $g(k)=0$ if and only if $k=0$, the degeneracy is exactly the same as the choice made in \cite{GLT} to obtain the porous medium equation with $m=2$. Namely, $c(x,x+e_i,\eta)=0$ both here and in the model of \cite{GLT} when $\eta(x-e_i)+\eta(x+2e_i)=0$, so the property is the same as in \cite{GLT}. Also we can provide for any other $m$ a proper choice of the rates such that all proofs can be readily extended leading in the diffusive re-scaling to the MPME with the correspondent $m$. For instance in the case $m=3$, the jump rates to be considered are
\begin{align}
c(x,x+e_j,\eta)&=g(\eta(x-e_j))g(\eta(x+2e_j))  \nonumber \\
&+g(\eta(x-2e_j))g(\eta(x-e_j)) +g(\eta(x+2e_j))g(\eta(x+3e_j)). \label{eq:m3jump}
\end{align}
For the choice proposed to obtain $m=3$, the degeneracy is also the same as in \cite{GLT}. Note that both the choices of the jump rates taken above have the property of defining a gradient system.

To prove the hydrodynamic behavior, we need some assumptions for the function $g$. First we state an assumption, which is usually required to prove the hydrodynamic behavior of the Zero Range process. Denote by $\psi^*$ the radius of convergence of the partition function $Z:\R_+ \to \R_+$ defined by
\begin{displaymath}
Z(\psi)=\sum_{k \ge 0}\frac{\psi^k}{g(k)!}
\end{displaymath} 
where $g(k)!= \Pi_{j=1}^kg(j)$ and $g(0)!=1$. Notice that $Z$ is analytic and strictly increasing on $[0, \psi^*)$. We assume for $g$ that $Z(\cdot)$ increases to $\infty$ as $\psi$ converges to $\psi^*$:
\begin{equation}
\lim_{\psi \uparrow \psi^*}Z(\psi)=\infty. \label{eq:infty}
\end{equation}
Now, we describe some invariant measures of this process which are also invariant for the Zero Range Process defined with the same function $g$. For each fixed $\psi \in [0,\psi^*)$, let $\bar{\nu}_\psi=\bar{\nu}_\psi^N$ denote a product measure on $\C_N^d$ with marginals given by
\begin{displaymath}
\bar{\nu}_\psi\{\eta(x)=k\} = \frac{\psi^k}{Z(\psi)g(k)!}
\end{displaymath}
for all $x \in \C_N^d$ and $k \in \N$. Then, the Markov process $\eta_t$ on $\chi_N^d$ is reversible with respect to the one parameter family of translation invariant product measures $\{\bar{\nu}_\psi\}_{\psi \in [0,\psi^*)}$.

Let $R(\psi)$ denote the expectation value of the occupation variable under $\bar{\nu}_\psi$, i.e., $R(\psi)=E_{\bar{\nu}_\psi}[\eta(0)]$. Under our assumption, it is known that $R:[0, \psi^*) \to \R_+$ is onto and one-to-one, so that there exists an inverse of $R$ (see e.g. Section 2 in \cite{KL}). Denote this inverse function by $\Phi$. By the definition, $\Phi$ is a smooth strictly increasing function and satisfies $\Phi(0)=0$ and $\lim_{\rho \to 0}\Phi^{\prime}(\rho)=g(1)$. Let $\nu_{\a}$ be the measure $\bar{\nu}_{\Phi(\a)}$. Then, the index $\a$ stands for the density of particles, namely $E_{\nu_\a}[\eta(0)]=\a$. A simple computation shows that 
\begin{equation}\label{eq:phidef}
\Phi(\a)=E_{\nu_{\a}}[g(\eta(0))].
\end{equation}
\begin{rem}\label{rem:moment}
By assumption (\ref{eq:infty}), for each $\a \in \R_+$ the measure $\nu_{\a}$ has a finite exponential moment:
there exists $\theta(\a)>0$ such that
\begin{displaymath}
E_{\nu_{\a}}[\exp(\theta \eta(0))] < \infty.
\end{displaymath}
\end{rem}
By the degeneracy of the rates, other invariant measures arise naturally. For example in one dimensional setting, any configuration $\eta$ such that the distance between the position of two consecutive nonempty sites is bigger than two has the exchange rates all of which vanish, because $c(x,x\pm1)g(\eta(x)) \neq 0$ only when $\eta(x) \{\eta(x\mp 1)+\eta(x\pm 2)\} \neq 0$. Therefore it is a blocked configuration and a Dirac measure supported on it is an invariant measure for this process. Since there are some blocked configurations, we need to study the irreducible components of the hyperplanes of configurations with a fixed number of particles in detail. 
\begin{rem}
Let $\Si_{N,k}$ denote the hyperplane of configurations with $k$ particles, namely
\begin{displaymath}
\Si_{N,k}=\{\eta \in \chi^d_N : \sum_{x \in \T^d_N}\eta(x)=k\}.
\end{displaymath}
For any pair of positive integers $N$ and $k$, $\Si_{N,k}$ is not irreducible. In fact, for example, a configuration $\eta \in \Si_{N,k}$ satisfying $\eta(x)=k$ for a single site $x \in \T^d_N$ and $\eta(y)=0$ for $y \neq x $ is a blocked configuration. Moreover, it is easily seen that a configuration is blocked when it does not contain at least a couple of sites at distance one or two with occupation number different from $0$. Note that for the case in \cite{GLT}, due to the exclusive constraints, there exists a constant $C(d) < \infty$ such that the hyperplane $\Si_{N,k}$ is irreducible for $k>C(d)(\frac{N}{3})^d$.
\end{rem}
\begin{rem}
Any two configurations $\eta$ and $\xi$ in $\Si_{N,k}$ belong to a same irreducible component if $\eta$ and $\xi$ have at least one $d$-dimensional hypercube of sites of linear size 2 with occupation number different from $0$. In other words, define the set $\Si_{N,k}^*$ as
\begin{displaymath}
\Si_{N,k}^*=\{\eta \in \Si_{N,k} : \sum_{x \in \T^d_N} \Pi_{y \in Q_x}\eta(y) \ge 1\}
\end{displaymath}
where $Q_x=\{y \in \T^d_N : y_i-x_i \in \{0,1\} \ \text{for all} \ 1 \le i \le d\}$, then $\Si_{N,k}^*$ is a subset of an irreducible component. This is a key ingredient to derive the hydrodynamic limit.

To show this, it is sufficient to see that a $d$-dimensional hypercube of particles of linear size 2 (i.e. $2^d$ particles which form a $d$-dimensional hypercube of linear size 2) is the mobile cluster, namely it has the following properties:
(i) there exists allowed sequence of jumps which allows to shift the mobile cluster to any other position, 
(ii) this allowed path is independent on the value of the occupation number on the remaining sites, (iii) the jump of any other particle to a neighboring site should be allowed when the mobile cluster is brought in a proper position in its vicinity. 
For the direct construction of the path in (i) and (ii), we refer the reader to \cite{GLT} where the path is described in Section 2. The property (iii) is easy to check.
\end{rem}
To prove Proposition \ref{prop:oneblock} we also assume the linear-growth of the power of $g$: 
\begin{displaymath}
(G) \quad \limsup_{k \to \infty} \frac{g(k)^2}{k} < \infty.
\end{displaymath}
\begin{rem}
If we consider the case $m >2$, we need to assume that
\begin{displaymath}
\limsup_{k \to \infty} \frac{g(k)^m}{k} < \infty.
\end{displaymath}
\end{rem}
\begin{rem}\label{rem:b}
Under the assumption (G), there exists some positive constant $b$ such that
\begin{displaymath}
g(k)^2 \le b \ k \quad \text{for all} \quad k \ge 0.
\end{displaymath}
\end{rem}
Let $\T^d$ denote the $d$-dimensional torus. Fix $\e>0$ and a initial profile $\rho_0:\T^d \to \R_+$ of class $C^{2+\e}(\T^d)$ satisfying the bounded condition, as the existence of a strictly positive constant $\delta_{0}$ such that  
\begin{equation}
\delta_{0} \le \rho_0(u) \quad \text{for all} \quad u \in \T^d. 
\end{equation}
Since $\rho_0$ is continuous, we can take $\delta _1>0$ as
\begin{equation}
\delta_{0} \le \rho_0(u) \le {\delta_1}  \quad \text{for all} \quad u \in \T^d. \label{eq:boundcondi}
\end{equation}
By the definition, $\Phi(\a)^2$ is a smooth strictly increasing function on $[\delta_0,\delta_1]$ and 
\begin{displaymath}
\sup_{\a \in [\delta_0,\delta_1]}|\Phi \cdot \Phi^{\prime}(\a)| < \infty.
\end{displaymath}
Therefore, by Theorem A2.4.1 of \cite{KL}, the equation (\ref{eq:mpme}) admits a solution that we denote by $\rho(t,u)$ which is of class $C^{1+\e, 2+\e}(\R_+ \times \T^d)$ and $\delta_0 \le \inf_{t,u}\rho(t,u) \le \sup_{t,u}\rho(t,u) \le {\delta_1}$.

Let $\nu_{\rho_0(\cdot)}^N$ be the product measure on $\chi_N^d$ such that:
\begin{displaymath}
\nu_{\rho_0(\cdot)}^N\{\eta,\eta(x)=k\}=\nu_{\rho_0(\frac{x}{N})}\{\eta,\eta(x)=k\}. 
\end{displaymath}
Hereafter, for $t \ge 0$, we denote by $\nu^N_{\rho_{(t,\cdot)}}$ the product measure on $\chi_N^d$ such that
\begin{displaymath}
\nu^N_{\rho_{(t,\cdot)}}\{\eta,\eta(x)=k\}=\nu_{\rho(t,\frac{x}{N})}\{\eta,\eta(x)=k\}. 
\end{displaymath}

For two measure $\mu$ and $\nu$ on $\chi_N^d$ denote by $H(\mu/\nu)$ the relative entropy of $\mu$ with respect to $\nu$, defined by:
\begin{displaymath}
H(\mu/\nu)=\sup_f \Big\{\int f d\mu - log \int e^f d\nu\Big\},
\end{displaymath}
where the supreme is carried over all continuous functions.

With these notations our main theorem is stated as follows:
\begin{thm}  \label{thm:main}
Under the assumption (G), let $\rho_0:\T^d \to \R_+$ be a initial profile of class $C^{2+\e}(\T^d)$ that satisfies the bounded condition (\ref{eq:boundcondi}) and $(\mu^N)_N$ be a sequence of probability measures on $\chi^N_d$ such that:
\begin{equation}
H(\mu^N/ \nu^N_{\rho_0(.)}) = o (N^d). \label{eq:initial}
\end{equation}
Then, for each $t \ge 0$
\begin{equation}
H(\mu^N S^N_t/ \nu^N_{\rho_{(t,\cdot)}}) = o (N^d),
\end{equation}
where $\rho(t,u)$ is a smooth solution of equation (\ref{eq:mpme}).
In the above formula, $S^N_t$ stands for the semigroup associated to the generator $L_N$ speeded up by $N^2$.
\end{thm}
To keep notation as simple as possible, hereafter we denote by $\mu^N_t$ the distribution on $\chi_N^d$ at macroscopic time $t$:
\begin{displaymath}
\mu^N_t:=\mu^N S^N_t,
\end{displaymath}
and by $\overline{\mu^N_t}$ the Cesaro mean of $\mu^N_t$:
\begin{displaymath}
\overline{\mu^N_t}:=\frac{1}{t}\int^t_0 \mu^N_s ds.
\end{displaymath}
\begin{rem} \label{rem:entropy}
Fix a bounded profile $\rho_0: \T^d \to \R_+$. In \cite{KL}, it is shown that every sequence of probability measures $\mu^N$ with entropy $H(\mu^N/\nu_{\rho_0(\cdot)}^N)$ of order $o(N^d)$ satisfies that
\begin{displaymath}
H(\mu^N/\nu_{\a}^N)=O(N^d)
\end{displaymath}
for every $\a >0$. In particular, if the entropy $H(\mu^N/\nu_{\a}^N)$ at time 0 is bounded by $C_0 N^d$, we have 
\begin{displaymath}
H(\overline{\mu^N_t}/\nu_{\a}^N) \le C_0 N^d \quad \text{for every} \quad t \ge 0 .
\end{displaymath}
\end{rem}

We can deduce the conservation of local equilibrium in the weak sense.
\begin{cor}
Under the assumption of Theorem \ref{thm:main}, for every continuous function $H:\T^d \to \R$, every bounded cylinder function $\Psi$ and every $t \ge 0$,
\begin{displaymath}
\lim_{N \to \infty}E_{\mu^N_t}[|\frac{1}{N^d}\sum_{x \in \T^d_N}H(\frac{x}{N})\tau_x \Psi(\eta)-\int_{\T^d}H(u)E_{\nu_{\rho(t,u)}}[\Psi]du|]=0
\end{displaymath}
where $\tau_x$ is the shift operator acting on the cylinder functions $f$ as well as configurations $\eta$ as follows:
\begin{displaymath}
\tau_xf(\eta)=f(\tau_x\eta), \quad (\tau_x\eta)(z):=\eta(z-x), \quad z \in \Z^d.
\end{displaymath}
\end{cor}

\section{Examples}
We present three classes of examples for $g:\N \to \R_+$ that satisfies both (\ref{eq:infty}) and the assumption (G).

\begin{exa}
Fix $q>0$ and let $g(k)$ be a real sequence:
\begin{displaymath}
g(0)=0, \quad  g(k)=\frac{k}{q+k-1} \quad \text{for all} \quad k \ge 1.
\end{displaymath}
It is well known that $Z(\psi)=(1-\psi)^{-q}$ and $\psi^*=1$. Furthermore by this explicit formula, we obtain that $\lim_{\psi \uparrow \psi^*}Z(\psi)=\infty$. The function $g$ also satisfies the assumption (G):
\begin{displaymath}
\lim_{k \to \infty}\frac{g(k)^2}{k}=\lim_{k \to \infty}\frac{k^2}{k(q+k-1)^2}=0.
\end{displaymath}
Therefore, we can apply Theorem \ref{thm:main} to the dynamics defined by $g$. In this case, $\Phi(\rho)$ and $D(\rho)$ also can be written explicitly:
\begin{displaymath}
\Phi(\rho)=\frac{\rho}{\rho+q} \quad D(\rho)=2\frac{\rho q}{(\rho+q)^3}.
\end{displaymath}
\end{exa}

\begin{exa}
Fix $0 \le \b \le 1$ and let $g(k)$ be a real sequence:
\begin{displaymath}
g(0)=0, \quad  g(1)=1, \quad g(k)=(\frac{k}{k-1})^{\b} \quad \text{for all} \quad k \ge 2.
\end{displaymath}
Then, $Z(\psi)=\sum_{k \ge0} \frac{\psi^k}{k^{\b}}$ and $\psi^*=1$. Furthermore it is well known that \linebreak  $\lim_{\psi \uparrow \psi^*}Z(\psi)=\infty$. The function $g$ also satisfies the assumption (G):
\begin{displaymath}
\lim_{k \to \infty}\frac{g(k)^2}{k}=\lim_{k \to \infty}\frac{k^{\b}}{k(k-1)^{\b}}=0.
\end{displaymath}
Therefore, we can apply Theorem \ref{thm:main} to the dynamics defined by $g$. The special case $\b=0$ is corresponding to Example 1 with $q=1$.
\end{exa}

\begin{exa}
Fix $0 < \gamma \le \frac{1}{2}$ and let $g(k)$ be a real sequence:
\begin{displaymath}
g(0)=0, \quad  g(k)=k^{\gamma} \quad \text{for all} \quad k \ge 1.
\end{displaymath}
Then, by $\lim_{k \to \infty}g(k)=\infty$, it is obvious that $\psi^*=\infty$ and $\lim_{\psi \uparrow \psi^*}Z(\psi)=\infty$. The function $g$ also satisfies the assumption (G):
\begin{displaymath}
\lim_{k \to \infty}\frac{g(k)^2}{k}=\lim_{k \to \infty}\frac{k^{2\gamma}}{k} < \infty.
\end{displaymath}
Therefore, we can apply Theorem \ref{thm:main} to the dynamics defined by $g$.
\end{exa}
\section{The Relative Entropy Method}
In this section, we prove Theorem \ref{thm:main} via the Relative Entropy Method due to Yau in \cite{Y}. The proof of Theorem \ref{thm:main} is divided in several lemmas. We start with introducing some notation. 
Fix $\a \in (0,\infty)$ and an invariant measure $\nu_\a$. Let
\begin{displaymath}
\psi^N_t  =  \frac{d\nu^N_{\rho(t,.)}}{d\nu_\alpha}, \
f^N_t =  \frac{d\mu^N_t}{d\nu_\alpha}, \
H_N(t) = H(\mu^N_t/ \nu^N_{\rho_{(t,\cdot)}}). 
\end{displaymath}
Since the measures $\nu^N_{\rho(t,.)}$ and $\nu_\alpha$ are product, it is very simple to obtain an expression for $\psi^N_t$ :
\begin{displaymath}
\psi^N_t  = \exp \{ \sum_{x \in \T^N_d}[\eta(x)  \log \frac{\Phi(\rho(t,\frac{x}{N}))}{\Phi(\alpha)} - \log \frac{Z(\Phi(\rho(t,\frac{x}{N}))}{Z(\Phi(\alpha))}]\}.
\end{displaymath}
We take $T>0$ arbitrarily and fix it in the rest of this paper. In order to prove the result, we are going to show that there exists a constant $\gamma >0$ satisfying
\begin{displaymath}
H_N(t) \le  o (N^d) + \frac{1}{\gamma} \int_0^t H_N(s)ds \quad \text{for all} \quad 0 \le t \le T 
\end{displaymath}
and apply Gronwall inequality to conclude.

There is a well-known estimate of the entropy production due to Yau \cite{Y}:
\begin{equation}
\partial_tH_N(t) \le \int_{\chi_N^d} \frac{1}{\psi^N_t}(N^2L_N^{*}\psi^N_t-\partial_t\psi^N_t)f^N_td\nu_\alpha^N \ \text{for all} \ t \ge 0, \label{eq:yau}
\end{equation}
where $L_N^{*}$ is the adjoint operator of $L_N$ in $L^2(\nu_\a)$. In our case, $L_N^*=L_N$.

By simple computations, we obtain that the term $\frac{N^2L_N^{*}\psi^N_t}{\psi^N_t}$ is bounded from above by
\begin{align*}
\sum_{x \in \T^N_d}\sum_{j=1}^d[\tau_xp_j(\eta)\partial_{u_j}^2\lambda(t,\frac{x}{N})+\frac{1}{2}\tau_xq_j(\eta)(\partial_{u_j}\lambda(t,\frac{x}{N}))^2]
+o(1)\sum_{x \in \T^N_d}\sum_{j=1}^d[|\tau_xp_j(\eta)|+\tau_xq_j(\eta)]
\end{align*}
where
\begin{displaymath}
p_j(\eta) = g(\eta(0))g(\eta(e_j))+g(\eta(0))g(\eta(-e_j))-g(\eta(e_j))g(\eta(-e_j)),
\end{displaymath}
\begin{displaymath}
q_j(\eta) = c(0,e_j,\eta)\{g(\eta(0)+g(\eta(e_j))\}=\{g( \eta(-e_j))+g(\eta(2e_j))\}\{g(\eta(0))+g(\eta(e_j))\},
\end{displaymath}
and
\begin{displaymath}
\lambda(t,u) = \log \Phi(\rho(t,u)).
\end{displaymath}
Notice that $W_{0,e_j}:=c(0,e_j,\eta)\{g(\eta(0)-g(\eta(e_j))\}=p_j(\eta)- \tau_{e_j}p_j(\eta)$ where $\tau_{e_j}$ is a shift operator.

Here and after, $o(1)$ means that the absolute value of the term is bounded from above by a constant $C_{N,T}$ depending only on $N$ and $T$ such that $\lim_{N \to \infty}C_{N,T}=0$. By Remark \ref{rem:b}
\begin{equation}
|\tau_xp_j(\eta)| \le b(\eta(x-e_j)+\eta(x)+\eta(x+e_j)) \label{eq:p}
\end{equation}
and
\begin{equation}
\tau_x q_j(\eta) \le b(\eta(x-e_j)+\eta(x)+\eta(x+e_j)+\eta(x+2e_j)) \label{eq:q}
\end{equation}
holds. Therefore, we obtain that
\begin{equation}
\frac{N^2L_N^{*}\psi^N_t}{\psi^N_t} \le \sum_{x \in \T^N_d}\sum_{j=1}^d[\tau_xp_j(\eta)\partial_{u_j}^2\lambda(t,\frac{x}{N})+\frac{1}{2}\tau_xq_j(\eta)(\partial_{u_j}\lambda(t,\frac{x}{N}))^2]
+o(1)\sum_{x \in \T^N_d}\eta(x). \label{eq:1}
\end{equation}
On the other hand, Taylor's expansion gives that
\begin{equation}
\sum_{x \in \T^N_d}\sum_{j=1}^d [\tilde{p}(\rho(t,\frac{x}{N}))\partial_{u_j}^2\lambda(t,\frac{x}{N})+\frac{1}{2}\tilde{q}(\rho(t,\frac{x}{N}))(\partial_{u_j}\lambda(t,\frac{x}{N}))^2] = o(N^d) \label{eq:2}
\end{equation}
where
\begin{displaymath}
\tilde{p} (\alpha):= E_{\nu_{\alpha}}[p_j(\eta)]=\Phi(\a)^2
\end{displaymath}
and
\begin{displaymath}
\tilde{q}(\alpha) := E_{\nu_{\alpha}}[q_j(\eta)]=4\Phi(\a)^2.
\end{displaymath}
By the identity 
\begin{displaymath}
\frac{Z^{\prime}(\psi)}{Z(\psi)}=\frac{R(\psi)}{\psi}
\end{displaymath}
and the fact that $\rho(t,u)$ is the solution of the equation (\ref{eq:mpme}), we can rewrite the term $\frac{1}{\psi^N_t}\partial_t\psi^N_t (\eta)=\partial_t(\log \psi^N_t)$ as
\begin{equation}
\sum_{x \in \T^N_d}\sum_{j=1}^d [\tilde{p}\prime (\rho(t,\frac{x}{N})) \partial_{u_j}^2\lambda(t,\frac{x}{N})
+ \frac{1}{2}\tilde{q} \prime (\rho(t,\frac{x}{N}))(\partial_{u_j}\lambda(t,\frac{x}{N}))^2](\eta(x)- \rho(t,\frac{x}{N})) \label{eq:3}.
\end{equation}
Up to this point we prove the next lemma:
\begin{lem}
For every $\e > 0$ and $l \in \N$, there exists $N_{\epsilon,l} \in \N$ such that for all $t \in [0,T]$ and $N \ge N_{\epsilon,l}$
\begin{align}
&\frac{1}{N^d} H^N(t) \le \e + \int_0^t ds \int_{\chi^N_d} \frac{1}{N^d}\sum_{x \in \T^N_d}\sum_{j=1}^d [ \partial_{u_j}^2\lambda(s,\frac{x}{N}) \{ \tau_xp_j(\eta) - \tilde{p} (\eta^l(x))\} \label{eq:lemma1} \\
& +\frac{1}{2}(\partial_{u_j}\lambda(t,\frac{x}{N}))^2 \{ \tau_xq_j(\eta) - \tilde{q} (\eta^l(x))\} \nonumber \\
&+ \partial_{u_j}^2\lambda(s,\frac{x}{N}) \{ \tilde{p} (\eta^l(x))- \tilde{p} (\rho(s,\frac{x}{N})) - \tilde{p}\prime (\rho(s,\frac{x}{N})) ((\eta^l(x) - \rho(s,\frac{x}{N})) \}  \nonumber \\
&+ \frac{1}{2}(\partial_{u_j}\lambda(s,\frac{x}{N}))^2 \{ \tilde{q} (\eta^l(x))- \tilde{q} (\rho(s,\frac{x}{N})) - \tilde{q}\prime (\rho(s,\frac{x}{N})) (\eta^l(x) - \rho(s,\frac{x}{N})) \}] f^N_s d\nu^N_{\alpha} \nonumber 
\end{align}
where $\eta^l(x)$ stands for the empirical density of particles in a cube of length $l$ centered at x:
\begin{displaymath}
\eta^l(x)=\frac{1}{(2l+1)^d}\sum_{|y-x| \le l}\eta(y).
\end{displaymath}
\end{lem}
\begin{proof}
From the equations (\ref{eq:yau}), (\ref{eq:1}), (\ref{eq:2}) and (\ref{eq:3}), we have only to prove the estimate:
\begin{align*}
\limsup_{N \to \infty}\ & |\int_0^t ds \int_{\chi^N_d} \frac{1}{N^d}\sum_{x \in \T^N_d}\sum_{j=1}^d \{ \partial_{u_j}^2\lambda(s,\frac{x}{N}) \tilde{p}\prime (\rho(s,\frac{x}{N})) \\
& + \frac{1}{2}(\partial_{u_j}\lambda(s,\frac{x}{N}))^2 \tilde{q}\prime (\rho(s,\frac{x}{N})) \} (\eta^l(x) - \eta(x))f^N_s d\nu^N_{\alpha}|=0
\end{align*}
for every $l \in \N$. Since $\partial_{u_j}^2\lambda(s,u) \tilde{p}\prime (\rho(s,u)) + \frac{1}{2}(\partial_{u_j}\lambda(s,u))^2 \tilde{q}\prime (\rho(s,u))$ is a uniformly continuous function on $[0,T] \times \T^d$, a summation by parts gives this estimate.
\end{proof}

To replace the cylinder functions $\tau_xp_j(\eta)$ and $\tau_xq_j(\eta)$ by their mean values $\tilde{p}(\eta^l(x))$ and $\tilde{q}(\eta^l(x))$ respectively, we need to prove the next proposition.
\begin{prop}[One-block Estimate]\label{prop:oneblock}
Let $\psi$ be $p_j$ or $q_j$. Then, for small $\gamma >0 $,
\begin{equation}
\limsup_{l \to \infty}\limsup_{N \to \infty}\int_0^t ds \int_{\chi^N_d} \frac{1}{N^d}\sum_{x \in \T^N_d} \tau_xV_{l,\psi}(\eta)f^{N}_s d\nu^N_{\alpha} \le \frac{1}{\gamma N^d}\int_0^t H^N(s) ds
\end{equation}
where 
\begin{displaymath}
V_{l,\psi}(\eta) = \lvert \frac{1}{(2l+1)^d} \sum_{|y| \le l} \tau_{y}\psi(\eta)-\tilde{\psi}(\eta^l(0)) \rvert
\end{displaymath}
and $\tilde{\psi}(\a)=E_{\nu_{\a}}[\psi]$. More precisely, for small $\gamma >0$ and every $\e >0 $, there exists $l_{\gamma, \e}$ such that for all $l \ge l_{\gamma, \e}$, there exists $N_{\gamma, \e, l}$ such that for all $t \in [0,T]$ and $N \ge N_{\gamma, \e,l}$,
\begin{displaymath}
\int_0^t ds \int_{\chi^N_d} \frac{1}{N^d}\sum_{x \in \T^N_d} \tau_xV_{l,\psi}(\eta)f^{N}_s d\nu^N_{\alpha} \le \e + \frac{1}{\gamma N^d}\int_0^t H^N(s) ds.
\end{displaymath}
\end{prop}
We prove this proposition in the next section. From a summation by parts and Proposition \ref{prop:oneblock}, we can deduce the replacement:
\begin{align*}
\limsup_{l \to \infty}\limsup_{N \to \infty}& \int_0^t ds \int_{\chi^N_d} \frac{1}{N^d}\sum_{x \in \T^N_d} [\partial_{u_j}^2\lambda(s,\frac{x}{N}) \{ \tau_xp_j(\eta) - \tilde{p} (\eta^l(x))\} \\
& +\frac{1}{2}(\partial_{u_j}\lambda(t,\frac{x}{N}))^2 \{ \tau_xq_j(\eta) - \tilde{q} (\eta^l(x))\}]f^{N}_s d\nu^N_{\alpha} \le \frac{1}{\gamma N^d} \int_0^t H^N(s) ds
\end{align*}
straightforwardly. The rigorous statement of this inequality is same as that of Proposition \ref{prop:oneblock}.

Next, to estimate the right hand side of (\ref{eq:lemma1}) we show that the expectation
\begin{displaymath}
\int_0^t ds \int_{\chi^N_d}\sum_{x \in \T^N_d}\sum_{j=1}^d [\partial_{u_j}^2\lambda(s,\frac{x}{N}) + 2(\partial_{u_j}\lambda(s,\frac{x}{N}))^2 ]M(\eta^l(x), \rho(s,\frac{x}{N})) f^N_s d\nu^N_{\alpha}
\end{displaymath}
is bounded from above by the sum of a term of $o(N^d)$ and the time integral of the entropy multiplied by a constant, where
\begin{align*}
M(a,b) &= \tilde{p}(a) - \tilde{p}(b) - \tilde{p} \prime (b)(a-b) \\
	&=\frac{1}{4}\{\tilde{q}(a) - \tilde{q}(b) - \tilde{q} \prime (b)(a-b)\}.
\end{align*}
By the entropy inequality, for every $\gamma >0$, this integral is bounded above by
\begin{align*}
\frac{1}{\gamma} \int_0^t H^N(s) ds + \frac{1}{\gamma} \int_0^t ds \log E_{\nu_{\rho(s,\cdot)}^N}[\exp\{\gamma\sum_{x \in \T^N_d}F(s,\frac{x}{N})M(\eta^l(x),\rho(s,\frac{x}{N}))\}]
\end{align*}
where
\begin{displaymath}
F(s,u)=\partial_{u_j}^2\lambda(s,u) + 2(\partial_{u_j}\lambda(s,u))^2.
\end{displaymath}
The next result concludes the proof of Theorem \ref{thm:main}.
\begin{prop}\label{prop:convex}
For sufficiently small $\gamma > 0$,
\begin{equation}
\limsup_{l \to \infty}\limsup_{N \to \infty} \frac{1}{\gamma N^d} \int_0^t ds \log E_{\nu_{\rho(s,\cdot)}^N}[\exp\{\gamma\sum_{x \in \T^N_d}F(s,\frac{x}{N})M(\eta^l(x),\rho(s,\frac{x}{N}))\}] \le 0
\end{equation}
for all $t \in [0,T]$.
\end{prop}
\noindent
The proof of this proposition is in Section 6.
\section{One-block Estimate}
In this section, we prove Proposition \ref{prop:oneblock}. We start with a key lemma.
\begin{lem}\label{lem:m0}
Fix $w > 0$ and $\a >0$. Then, there exists a constant $M_0 > 0$ and $\gamma_0 >0$ such that for all $\gamma \le \gamma_0$,
\begin{displaymath}
\lim_{n \to \infty} \frac{1}{n} \log E[\exp(w \gamma \sum_{k=1}^n X_k \cdot 1_{\{\bar{X_n} > M_0\}})] \le 0
\end{displaymath}
where $\{X_k\}_{k=1}^{\infty}$ is a sequence of i.i.d random variables with distribution $\nu_{\a}$ and $\bar{X_n}=\sum_{k=1}^n X_k/n$. 
\end{lem}
\begin{proof}
By Remark \ref{rem:moment}, the set $\{ \theta >0; E[\exp(\theta X_1)] < \infty \}$ is not empty. Take an element $\theta$ from the set and fix it. We prove the statement for $\gamma_0= \frac{\theta}{2w}$. For every $\gamma >0 $ satisfying $\gamma \le \gamma_0$, namely $2w\gamma \le \theta$,
\begin{align*}
\log E[\exp(w & \gamma \sum_{k=1}^n X_k  \cdot 1_{\{\bar{X_n} > M_0\}})] \le \log \Big( E[\exp(w \gamma \sum_{k=1}^n X_k) \cdot 1_{\{\bar{X_n} > M_0\}}] +1 \Big) \\
& \le  E[\exp(w \gamma \sum_{k=1}^n X_k) \cdot 1_{\{\bar{X_n} > M_0\}}]  \le E[\exp(2w \gamma \sum_{k=1}^n X_k)]^{1/2} E[1_{\{\bar{X_n} > M_0\}}]^{1/2} \\
& = E[\exp(2w \gamma X_1)]^{n/2} E[1_{\{\bar{X_n} > M_0\}}]^{1/2}. 
\end{align*}
A simple computation shows that
\begin{align*}
E[1_{\{\bar{X_n} > M_0\}}] = E[1_{\{ \theta \sum_{k=1}^n X_k >\theta n M_0\}}] 
\le E[\exp( \theta \sum_{k=1}^n X_k - \theta n M_0)] = \exp(- \theta n M_0)E[\exp(\theta X_1)]^n.
\end{align*}
Denote $\log E[\exp(aX_1)]$ by $R(a)$, then
\begin{align*}
\frac{1}{n} &\log E[\exp(w \gamma \sum_{k=1}^n X_k \cdot 1_{\{\bar{X_n} > M_0\}})] \le \frac{1}{n} \exp(R(2w\gamma)\frac{n}{2}) \exp(- \theta M_0 \frac{n}{2}) \exp(R(\theta)\frac{n}{2})\\
& = \frac{1}{n} \exp(\frac{n}{2}\{R(2w\gamma)- \theta M_0+ R(\theta)\}) \le \frac{1}{n} \exp[\frac{n}{2}\{- \theta M_0+ 2R(\theta)\}].
\end{align*}
Therefore, we choose $M_0$ as $M_0 > \frac{2R(\theta)}{\theta}$ and conclude the proof.
\end{proof}

Next, we show that this lemma allows us to introduce an indicator function the same way as the proof of Zero Range Process in \cite{KL}. First, recall the definition of the function $\tau_x V_{l,\psi}$:
\begin{displaymath}
\tau_xV_{l,\psi}(\eta) = \lvert \frac{1}{(2l+1)^d} \sum_{|y-x| \le l} \tau_{y}\psi(\eta)-\tilde{\psi}(\eta^l(x)) \rvert.
\end{displaymath}
Since we assume $\psi$ is $p_j$ or $q_j$ and we have the estimates (\ref{eq:p}) and (\ref{eq:q}), $\psi$ satisfies that $\psi$ is measurable with respect to $\{\eta(y); |y| \le A\}$ and $|\psi(\eta)| \le b \sum_{|y| \le  A}\eta(y)$ for some finite constant $A$. Therefore, simple computations show that
\begin{align*}
& |\frac{1}{(2l+1)^d} \sum_{|y-x| \le l} \tau_{y}\psi(\eta)| \le \frac{b}{(2l+1)^d} \sum_{|y-x| \le l}\sum_{|z-y| \le A}\eta(z) \\
&\le \frac{b(2A+1)^d}{(2l+1)^d} \sum_{|z-x| \le l+A}\eta(z) 
 \le  \frac{b(2A+1)^d(2l+2A+1)^d}{(2l+1)^d} \eta^{l+A}(x) \le b^{\prime}\eta^{l+A}(x)
\end{align*}
for some finite constant $b^{\prime}$ for every $l$. On the other hand, from the estimate
\begin{displaymath}
|\psi(\a)|=|E_{\nu_{\a}}[\psi(\eta)]| \le b E_{\nu_{\a}}[\sum_{|y| \le  A}\eta(y)] \le b(2A+1)^d \a,
\end{displaymath}
we have the inequality
\begin{displaymath}
|\psi(\eta^l(x))| \le b(2A+1)^d\eta^l(x) \le \frac{b(2A+1)^d(2l+2A+1)^d}{(2l+1)^d} \eta^{l+A}(x) \le b^{\prime}\eta^{l+A}(x).
\end{displaymath}
These two inequality lead to that:
\begin{equation}
\tau_xV_{l,\psi}(\eta) \le w \eta^{l+A}(x) \label{eq:v}
\end{equation}
where $w=2b^{\prime}$. Here, we take $M_0>0$ and $\gamma_0>0$ for this $w$ and $\delta_1$ whose existence is guaranteed in Lemma \ref{lem:m0}. Then, we can divide the integral that appears in the statement of Proposition \ref{prop:oneblock} into
\begin{align}
&\int_0^t ds \int_{\chi^N_d} \frac{1}{N^d}\sum_{x \in \T^N_d} \tau_xV_{l,\psi}(\eta)  1_{\{\eta^{l+A}(x) \le M_0\}} f^{N}_s d\nu^N_{\alpha} \label{eq:bdterm}\\
&+ \int_0^t ds \int_{\chi^N_d}\frac{1}{N^d}\sum_{x \in \T^N_d} \tau_xV_{l,\psi}(\eta) 1_{\{\eta^{l+A}(x) > M_0\}} f^{N}_s d\nu^N_{\alpha} \label{eq:unbdterm}.
\end{align}
By the inequality (\ref{eq:v}) and the entropy inequality, the term (\ref{eq:unbdterm}) is bound from above by
\begin{align*}
& \int_0^t ds \int_{\chi^N_d}\frac{1}{N^d}\sum_{x \in \T^N_d} w \eta^{l+A}(x) 1_{\{\eta^{l+A}(x) > M_0\}} f^{N}_s d\nu^N_{\alpha} \\
& \le \frac{1}{\gamma N^d} \int_0^t H^N(s) ds + \frac{1}{\gamma N^d} \int_0^t ds \log E_{\nu_{\rho(s,\cdot)}^N}[\exp\{ \gamma w\sum_{x \in \T^N_d} \eta^{l+A}(x) 1_{\{\eta^{l+A}(x) > M_0\}}\}]
\end{align*}
for every $\gamma >0$.
By H\"{o}lder inequality and by independence, the second term of last expression can be bounded by
\begin{align*}
&\frac{1}{\gamma N^d(2l+2A+1)^d} \sum_{x \in \T^N_d} \int_0^t ds \log E_{\nu_{\rho(s,\cdot)}^N}[\exp\{ \gamma w  \sum_{|z-x| \le l+A}\eta(z) 1_{\{\eta^{l+A}(x) > M_0\}}\}] \\
& \le \frac{1}{\gamma N^d(2l+2A+1)^d} \sum_{x \in \T^N_d} \int_0^t ds \log E_{\nu_{\delta_1}}[\exp\{ \gamma w  \sum_{|z-x| \le l+A}\eta(z) 1_{\{\eta^{l+A}(x) > M_0\}}\}] \\
& = \frac{t}{\gamma (2l+2A+1)^d}  \log E_{\nu_{\delta_1}}[\exp\{ \gamma w  \sum_{|z| \le l+A}\eta(z) 1_{\{\eta^{l+A}(0) > M_0\}}\}]
\end{align*}
which vanishes as $l \to \infty$ for every $\gamma \le \gamma_0$ by Lemma \ref{lem:m0}. Here we use the fact $\rho(s,u) \le \delta_1$ for every $s \ge 0$ and $u \in \T^d$.

Now, we deal with the term (\ref{eq:bdterm}). We separate the set of configurations into two sets: the  irreducible component that contains all configurations with at least one $d$-dimensional hypercube of particles of linear size 2 and the remaining configurations. In the first case the standard proof is easily adapted, while for the second case we will use the fact that this set has small measure with respect to $\nu_{\rho(s,\cdot)}^N$.

Fix $x \in \T^d_N$ and denote by $\mathcal{Q}_{x,l}$ the set of configurations which have at least one $d$-dimensional hypercube of sites of linear size 2 with occupation number different from $0$ in the box center $x$ and radius $l$:
\begin{displaymath}
\mathcal{Q}_{x,l}= \Big\{ \eta \ : \sum_{y \in C^x} \Pi_{z \in Q_y} \eta(z) \ge 1 \Big\}
\end{displaymath}
where $Q_y=\{z : z_i-y_i \in \{0,1\} \ \text{for all} \ 1 \le i \le d\}$ and $C^x= \{y : |z-x| \le l \ \text{for all} \ z \in Q_y\}$. We denote by $\mathcal{E}_{x,l}$ the irreducible set which contains $\mathcal{Q}_{x,l}$ (and all configurations that can be connected via an allowed path to one in $\mathcal{E}_{x,l}$) and we split the term (\ref{eq:bdterm}) into
\begin{align}
&\int_0^t ds \int_{\chi^N_d} \frac{1}{N^d}\sum_{x \in \T^N_d} \tau_xV_{l,\psi}(\eta) 1_{\{\mathcal{E}_{x,l}\}}(\eta)
 1_{\{\eta^{l+A}(x) \le M_0\}} f^{N}_s d\nu^N_{\alpha} \label{eq:bdterm1}\\
&+ \int_0^t ds \int_{\chi^N_d}\frac{1}{N^d}\sum_{x \in \T^N_d} \tau_xV_{l,\psi}(\eta) 1_{\{\mathcal{E}_{x,l}^c \}}(\eta) 1_{\{\eta^{l+A}(x) \le M_0\}} f^{N}_s d\nu^N_{\alpha} \label{eq:bdterm2}.
\end{align}
For the term (\ref{eq:bdterm1}), we can repeat the standard argument of the One-block estimate with Remark \ref{rem:entropy} because we have already succeeded to cut off large densities and we conclude:
\begin{displaymath}
\limsup_{l \to \infty}\limsup_{N \to \infty}\int_0^t ds \int_{\chi^N_d} \frac{1}{N^d}\sum_{x \in \T^N_d} \tau_xV_{l,\psi}(\eta) 1_{\{\mathcal{E}_{x,l}\}}(\eta)
 1_{\{\eta^{l+A}(x) \le M_0\}} f^{N}_s d\nu^N_{\alpha}=0.
\end{displaymath}
It remains to show that the term (\ref{eq:bdterm2}) is bounded from above by the sum of a term which vanishes as $N \to \infty$ and the time integral of the entropy multiplied by a constant and divided by $N^d$. Denote the probability of the set $\{ \eta; \eta(0) \ge 1\}$ under $\nu_{\delta_0}$ by $P_{\delta_0}$:
\begin{displaymath}
P_{\delta_0}=\nu_{\delta_0}(\{ \eta; \eta(0) \ge 1\}).
\end{displaymath}
Note that the probability of the ergodic set $\mathcal{E}_{x,l}$ converges rapidly to one with $l$, indeed the following holds:
\begin{equation}
\nu_{\rho(s, \cdot)}^N(\mathcal{E}_{x,l}) \ge 1-(1-P_{\delta_0}^{2^d})^{l^d} \label{eq:qxl}
\end{equation}
where we use the fact that initial profile is bounded away from zero, namely $\rho(s,u) \ge \delta_0$ for every $u \in \T^d$. 

The equation (\ref{eq:v}) shows that we can bound the term (\ref{eq:bdterm2}) from above by
\begin{displaymath}
\int_0^t ds \int_{\chi^N_d}\frac{1}{N^d}\sum_{x \in \T^N_d} w M_0 1_{\{\mathcal{E}_{x,l}^c \}}(\eta)f^{N}_s d\nu^N_{\alpha}
\end{displaymath}
and the entropy inequality allow us to bound it by
\begin{displaymath}
\frac{1}{\gamma N^d} \int_0^t H^N(s) ds + \frac{1}{\gamma N^d} \int_0^t ds \log E_{\nu_{\rho(s,\cdot)}^N}[\exp\{ \gamma w M_0 \sum_{x \in \T^N_d} 1_{\{\mathcal{E}_{x,l}^c \}} \}]
\end{displaymath}
for every $\gamma >0$.
By H\"{o}lder inequality and by independence, the second term of the last expression can be bounded from above by
\begin{align*}
&\frac{1}{\gamma N^d (2l+1)^d}  \sum_{x \in \T^N_d} \int_0^t ds \log E_{\nu_{\rho(s,\cdot)}^N}[\exp\{ \gamma w M_0 (2l+1)^d 1_{\{\mathcal{E}_{x,l}^c \}} \}] \\
&=\frac{1}{\gamma N^d (2l+1)^d}  \sum_{x \in \T^N_d} \int_0^t ds \log \Big( \nu_{\rho(s,\cdot)}^N(\mathcal{E}_{x,l}^c)(\exp\{ \gamma w M_0 (2l+1)^d\}-1)+1 \Big).
\end{align*}
By using the upper bound on $\nu_{\rho(s,\cdot)}^N(\mathcal{E}_{x,l}^c)$ which follows from (\ref{eq:qxl}) and the inequality $\log(x+1) \le x$, we can bound from above the last expression by 
\begin{displaymath}
\frac{t}{\gamma (2l+1)^d} (\exp\{ \gamma w M_0 (2l+1)^d\}-1) (1-P_{\delta_0}^{2^d})^{l^d},
\end{displaymath}
which vanishes as $l \to \infty$ for sufficiently small $\gamma>0$ since $P_{\delta_0}>0$.

\section{Proof of Proposition \ref{prop:convex}}
In this section, we prove Proposition \ref{prop:convex}. First, we reduce the problem stated in proposition to the problem where the time is fixed. With the elementary estimations, we obtain that
\begin{align*}
|M(\eta^l(x),\rho(s,\frac{x}{N}))| & \le \tilde{p}(\eta^l(x)) + \tilde{p}(\delta_1) + \sup_{a \in [\delta_0, {\delta_1} ] }\tilde{p}^{\prime} (a) \eta^l(x) + \sup_{a \in [\delta_0, {\delta_1} ] }
 a\tilde{p}^{\prime} (a) \\
& \le C_1 \eta^l(x) + C_2 
\end{align*} 
for some constants $C_1$ and $C_2$ because $\tilde{p}(\la)= \Phi(\la)^2=E_{\nu_{\la}}[g(\eta(0))g(\eta(e_1))] \le 1/2 E_{\nu_{\la}}[g(\eta(0))^2 + g(\eta(e_1))^2] \le b \la$. Therefore, by H\"{o}lder inequality and by independence, for sufficiently small $\gamma > 0$
\begin{align*}
&\frac{1}{\gamma N^d} \log E_{\nu_{\rho(s,\cdot)}^N}[\exp\{\gamma\sum_{x \in \T^N_d}F(s,\frac{x}{N})M(\eta^l(x),\rho(s,\frac{x}{N}))\}] \\
& \le \frac{1}{\gamma N^d (2l+1)^d}\sum_{x \in \T^N_d}\log E_{\nu_{\rho(s,\cdot)}^N}[\exp\{\gamma \|F\|_{\infty}(2l+1)^d( C_1 \eta^l(x) + C_2)\}] \\
& \le \frac{1}{\gamma N^d (2l+1)^d}\sum_{x \in \T^N_d}\log \Big\{  \exp(\gamma \|F\|_{\infty}(2l+1)^d C_2) E_{\nu_{{\delta_1}}}[\exp\{\gamma \|F\|_{\infty}C_1 \eta(0)\}]^{(2l+1)^d} \Big\} \\
& = \|F\|_{\infty}C_2 + \frac{1}{\gamma} \log E_{\nu_{{\delta_1}}}[\exp\{\gamma \|F\|_{\infty}C_1 \eta(0)\}] := k < \infty.
\end{align*}
In this formula, $\|F\|_{\infty}$ stands for the $L^{\infty}([0,T] \times \T^d)$ norm of $F$:
\begin{displaymath}
\|F\|_{\infty}= \sup_{(s,u) \in [0,T] \times \T^d}|F(s,u)|
\end{displaymath}
By the definition, $k$ does not depend on $N$, $l$ nor $s$. Therefore, we can apply Fatou's lemma to bound the expression
\begin{displaymath}
\limsup_{l \to \infty}\limsup_{N \to \infty} \int^t_0 \frac{1}{\gamma N^d}
 \log E_{\nu_{\rho(s,\cdot)}^N}[\exp\{\gamma\sum_{x \in \T^N_d}F(s,\frac{x}{N})M(\eta^l(x),\rho(s,\frac{x}{N}))\}] 
\end{displaymath}
from above by 
\begin{displaymath}
\int^t_0 \limsup_{l \to \infty}\limsup_{N \to \infty} \Big(\frac{1}{\gamma N^d}  \log E_{\nu_{\rho(s,\cdot)}^N}[\exp\{\gamma\sum_{x \in \T^N_d}F(s,\frac{x}{N})M(\eta^l(x),\rho(s,\frac{x}{N}))\}] \Big).
\end{displaymath}

The next lemma concludes the proof of Proposition \ref{prop:convex}.
\begin{lem}
There exists $\gamma_1 > 0 $ such that for all $0 \le s \le T$
\begin{equation}
\limsup_{l \to \infty}\limsup_{N \to \infty} \frac{1}{N^d} \log E_{\nu_{\rho(s,\cdot)}^N}[\exp\{\gamma_1 \sum_{x \in \T^N_d}F(s,\frac{x}{N})M(\eta^l(x),\rho(s,\frac{x}{N}))\}] \le 0.
\end{equation}
\end{lem}
To prove this lemma, we use the statement of Lemma 6.1.8 in \cite{KL}:
\begin{lem}
Let $G:\T^d \times \R_+ \to \R$ be a continuous function such that
\begin{equation}
\sup_{u \in \T^d}|G(u,\la)| \le D_0 + D_1 \la \quad \text{for all} \ \la \in \R_+
\end{equation}
where $D_0$ is a finite constant and $D_1$ is a constant bounded by $\log[\psi^*/\Phi(\delta_1)]$:
\begin{displaymath}
D_1 < \log \frac{\psi^*}{\Phi(\delta_1)}.
\end{displaymath}
Then,
\begin{align*}
\limsup_{l \to \infty}\limsup_{N \to \infty}& \frac{1}{N^d} \log E_{\nu_{\rho(\cdot)}^N}[\exp\{ \sum_{x \in \T^N_d}G \big(\frac{x}{N}, \eta^l(x) \big)\}] \\
& \le \int_{\T^d} du \sup_{\la \ge 0} \{G(u,\la)-J_{\rho(u)}(\la)\},
\end{align*}
where $J_{\b}(\cdot)$ is a rate function:
\begin{displaymath}
J_{\b}(\la)=\begin{cases}
		\la \log (\frac{\Phi(\la)}{\Phi(\b)}) - \log \Big(\frac{Z(\Phi(\la))}{Z(\Phi(\b))}\Big) & \text{for} \ \la \ge 0 \\
		\infty & \text{otherwise}.
\end{cases}
\end{displaymath}
\end{lem}
To apply this lemma to the function
\begin{displaymath}
G(u,\la)=\gamma F(s,u)\{\tilde{p}(\la) - \tilde{p}(\rho(s,u)) - \tilde{p}^{\prime} (\rho(s,u))(\la-\rho(s,u))\},
\end{displaymath}
notice that
\begin{displaymath}
\sup_{u \in \T^d}|G(u,\la)| \le \gamma \|F\|_{\infty} \{b \la +\sup_{a \in [\delta_0, \delta_1]} \tilde{p}(a) + \sup_{a \in [\delta_0, \delta_1]}\tilde{p}^{\prime} (a) \la + \sup_{a \in [\delta_0, \delta_1]} a\tilde{p}^{\prime} (a)\}.
\end{displaymath}

We summarize the conclusions up to this point in the next corollary.
\begin{cor}
Let 
\begin{displaymath}
\gamma_2= \frac{1}{(b+ \sup_{a \in [\delta_0, \delta_1]}\tilde{p}^{\prime} (a))\|F\|_{\infty}} \log \frac{\psi^*}{\Phi(\delta_1)}.
\end{displaymath}
Then, for all $\gamma < \gamma_2$ and $0 \le s \le T$,
\begin{align*}
\limsup_{l \to \infty}\limsup_{N \to \infty}& \frac{1}{N^d} \log E_{\nu_{\rho(s, \cdot)}^N}[\exp\{\gamma \sum_{x \in \T^N_d}F(s,\frac{x}{N})M(\eta^l(x), \rho(s,\frac{x}{N}))\}] \\
& \le \int_{\T^d} du \sup_{\la \ge 0} \{ \gamma F(s,u)M(\la,\rho(s,u))-J_{\rho(s,u)}(\la)\}.
\end{align*}
\end{cor}

To conclude the proof of Proposition \ref{prop:convex}, we have to show that the right hand side of the previous inequality is non positive for all $\gamma$ sufficiently small. This result follows from the next lemma.
\begin{lem}\label{lem:mjratio}
For every $0 < K_1 < K_2 < \infty$,
\begin{displaymath}
\sup_{\b \in [K_1,K_2], \la \ge 0} \frac{|M(\la, \b)|}{J_{\b}(\la)} < \infty.
\end{displaymath}
\end{lem}
\begin{proof}
Straightforward form the proof of Lemma 6.1.10 in \cite{KL} with the fact that $\tilde{p}(\la) \le b \la$.
\end{proof}
\begin{cor}
There exists $\gamma_1>0$ such that for all $\gamma < \gamma_1$
\begin{displaymath}
\sup_{(s,u) \in [0,T] \times \T^d, \lambda \ge 0} \{ \gamma F(s,u) M(\la, \rho(s,u)) -J_{\rho(s,u)}(\la)\} \le 0.
\end{displaymath}
\end{cor}
\begin{proof}
Straightforward from Lemma \ref{lem:mjratio} because $F$ is bounded on $[0,T] \times \T^d$ and the range of $\rho(\cdot, \cdot)$ is contained in $[\delta_0, \delta_1]$.
\end{proof}

\subsection*{Acknowledgement}
The author would like to thank Professor T.Funaki for helping her with valuable suggestions.

\end{document}